\documentclass[12pt]{article}
\usepackage[T1]{fontenc}
\usepackage[utf8]{inputenc}
\usepackage{amsmath,amssymb,amsthm}
\usepackage{mathptmx}
\usepackage[margin=1in]{geometry}
\usepackage{microtype,needspace}
\usepackage{tikz}
\usepackage[bottom]{footmisc}
\usepackage[numbers,sort&compress]{natbib}
\usepackage{xurl}
\usepackage[hidelinks]{hyperref}
\hypersetup{
 pdftitle={Singular Loci of SL2(C)-Representation Varieties of Torus Knot Groups},
 pdfauthor={S. Liriano},
 pdfsubject={Version 8: representation and character singularities, 4 November 2026}
}
\newtheorem{theorem}{Theorem}
\newtheorem{lemma}[theorem]{Lemma}
\newtheorem{proposition}[theorem]{Proposition}
\newtheorem{corollary}[theorem]{Corollary}
\newtheorem{sublemma}{Sublemma}[theorem]
\theoremstyle{remark}
\newtheorem*{remark}{Remark}
\newcommand{\C}{\mathbb C}
\newcommand{\Z}{\mathbb Z}
\newcommand{\RR}{\mathbb R}
\newcommand{\SL}{\mathrm{SL}_2(\C)}
\newcommand{\Sing}{\operatorname{Sing}}
\newcommand{\Hom}{\operatorname{Hom}}
\newcommand{\tr}{\operatorname{tr}}
\newcommand{\diag}{\operatorname{diag}}
\newcommand{\cI}{\mathcal I}
\newcommand{\Cc}{\mathcal C}
\newcommand{\fnsep}{\nobreak\hspace{0.3em}}
\title{\bfseries Singular Loci of $\mathrm{SL}_2(\mathbb C)$-Representation\\
Varieties of Torus Knot Groups}
\author{S. Liriano}
\date{}

\begin{document}
\maketitle
\vspace{-1.5em}

\begin{abstract}
Let $T(q,p)=\langle x,y\mid x^p=y^q\rangle$, where $p,q>1$ are relatively prime, and let
\[
R=\Hom(T(q,p),\mathrm{SL}_2(\mathbb C)).
\]
We compute the integral homology of the singular locus of $R$. We show that this locus has complex dimension $2$ and is a disjoint union of $\mu=(p-1)(q-1)$ smooth affine surfaces, each homotopy equivalent to the two-dimensional sphere. Consequently,
\[
H_0(\Sing R;\Z)\cong H_2(\Sing R;\Z)\cong\Z^\mu,
\]
and all remaining homology groups vanish. Thus both nonzero Betti numbers equal the first Betti number of the Milnor fiber of the plane curve singularity $z^p+w^q=0$, namely its Milnor number $\mu$.
\end{abstract}

\smallskip
\noindent\textbf{Research development and use of tools.}
This work grew out of several years of investigation and discussion concerning singularities of representation varieties in fg groups. Exploratory computations and code development using SageMath, including its GAP capabilities, Singular, and Macaulay2 were undertaken predominantly over finite fields to develop examples and geometric intuition. The local computing environment included a laptop running Linux. These investigations helped motivate the subsequent focus on the singular locus over the complex numbers. OpenAI assisted with mathematical discussions, the development and examination of proposed arguments, computational code, and the drafting and repeated revision of theorem statements, and exposition. It also assisted with preparation of the \LaTeX{} source and rendering of the manuscript. Throughout this process, the statements and proofs underwent extensive, repeated review and revision. The author takes full responsibility for the mathematical content.
\smallskip

The paper is organized as follows. Section~\ref{sec:components} describes the irreducible components of $R$ and their intersections. Section~\ref{sec:singular} identifies the singular locus and determines its profile, homotopy type, and integral homology. Section~\ref{sec:milnor} relates the resulting Betti numbers to the Milnor fiber and discusses the trefoil knot example.


\section{Introduction}

Let $p,q>1$ be relatively prime integers and put
\[
T(q,p)=\langle x,y\mid x^p=y^q\rangle,\qquad
R=\Hom(T(q,p),\SL).
\]
Evaluation on the generators identifies $R$ with the affine variety
$\{(A,B)\in\SL^2:A^p=B^q\}$. Throughout, varieties and singular loci carry their reduced structures.\fnsep\footnote{Thus their coordinate rings contain no nonzero nilpotent elements. We study the reduced affine variety underlying the representation scheme.} Dimension means complex dimension.\fnsep\footnote{For a complex affine variety $X$, its complex analytic dimension equals the Krull dimension of its coordinate ring $\C[X]$. See Serre \cite[\S2, no.~6, Corollaries 2--3, p.~11]{Serre1956}.} Homotopy types and homology refer to the usual topology on complex points, and homology has integral coefficients. We use $\cong$ for isomorphisms of varieties or groups, as appropriate, and $\simeq$ for homotopy equivalence. We write
\[
g=\frac{(p-1)(q-1)}2,\qquad \mu=(p-1)(q-1)=2g.
\]

Liriano showed that $R$ has dimension $4$ and that its number of four-dimensional irreducible components equals the genus $g$ of the corresponding torus knot.\fnsep\footnote{See \cite{Liriano1998,Liriano1999}; the genus equality is a consequence of the component formula in \cite[Theorem A]{Liriano1999}.}
Mart\'in-Morales and Oller-Marc\'en subsequently completed the component count for the family $\langle x,y\mid x^m=y^n\rangle$, including the components of lower dimension.\fnsep\footnote{Both \cite{Liriano1999} and \cite{MartinOller2010} allow exponents that are not relatively prime. In the coprime case there is one three-dimensional component and $g$ four-dimensional components.}
Later, Mira-Alban\'es, Rodr\'iguez-Nieto, and Salazar-D\'iaz gave explicit matrix descriptions of abelian and nonabelian $\mathrm{SL}_2(\C)$-representations of torus knot groups, including criteria for the irreducibility of individual representations.\fnsep\footnote{See \cite[Proposition 4 and Theorems 8--9]{Mira2023}.}

For the character variety, Mu\~noz described the abelian component and the $g$ additional components, each an affine line meeting the abelian component at two distinct points.\fnsep\footnote{See \cite[Proposition 2.3 and Theorem 3.1]{Munoz2009}. These results concern the character variety, the affine quotient of $R$ by simultaneous conjugation.}
Heusener, Porti, and Su\'arez Peir\'o established a general local description of the representation variety at diagonal representations associated with simple roots of the Alexander polynomial.\fnsep\footnote{For knot exteriors, \cite[Corollary 1.3 and the following discussion]{HeusenerPortiSuarez2001} identifies a smooth two-dimensional conjugacy orbit in the intersection of the abelian and nonabelian components. The roots of the Alexander polynomial of a torus knot are simple.}
Here we give a direct global account of the singular locus for torus knots and deduce its homotopy type and integral homology. 

The profile function records the dimensions and numbers of irreducible components.\fnsep\footnote{See Liriano and Majewicz \cite{LirianoMajewicz2011}. The profile counts maximal irreducible components of the variety.}
For a variety $X$ of dimension $d$, write $P_d(X)=(N_d(X),N_{d-1}(X),\ldots,N_0(X))$, where $N_j(X)$ is the number of irreducible components of $X$ of dimension $j$.
Thus, the preceding component count gives
\begin{equation}\label{eq:R-profile}
P_4(R)=(g,1,0,0,0).
\end{equation}
In the sequel we determine the corresponding profile and integral homology of the singular locus.

The geometry needed for this calculation is elementary. The abelian component is isomorphic to $\SL$, and each four-dimensional component is isomorphic to $Q\times Q$.\fnsep\footnotemark{} Here
\begin{equation}\label{eq:Q}
Q=\{U\in\mathfrak{sl}_2(\C):U^2=I\}\cong\SL/D.
\end{equation}
\footnotetext{The Lie algebra $\mathfrak{sl}_2(\C)$ consists of the $2\times2$ complex matrices of trace zero, with bracket $[U,V]=UV-VU$. Every element has the form $U=\left(\begin{smallmatrix}a&b\\c&-a\end{smallmatrix}\right)$, $a,b,c\in\C$. Since $U^2=(a^2+bc)I$, the condition $U^2=I$ identifies $Q$ with the affine quadric $\{(a,b,c)\in\C^3:a^2+bc=1\}$. Equivalently, $Q$ is the conjugacy class of $\diag(1,-1)$ under $\SL$. Its stabilizer is the diagonal torus $D=\{\diag(t,t^{-1}):t\in\C^\times\}$, giving $Q\cong\SL/D$. For a complex algebraic group $M$ and a closed algebraic subgroup $H$, one has $\dim(M/H)=\dim M-\dim H$. Here $\dim\SL=3$ and $\dim D=1$, so $\dim Q=2$.}
Distinct four-dimensional components are disjoint, and each meets the abelian component in two disjoint copies of $Q$. All components of $R$ are smooth. The singular locus therefore consists exactly of their intersections, and
\[
P_2(\Sing R)=(2g,0,0).
\]
We prove that $Q$ is diffeomorphic to the tangent bundle of $S^2$. It follows that $\Sing R$ has the homotopy type of $2g$ disjoint two-spheres. Although its points are singular in $R$, the variety $\Sing R$ itself is smooth.

The singularities studied here occur in the full representation variety, before taking the quotient by simultaneous conjugation. Each connected component of $\Sing R$ is a closed conjugacy orbit isomorphic to $Q\cong\SL/D$, and the character quotient maps each such surface to an ordinary double point.\fnsep\footnote{The orbit description follows from Lemma~\ref{lem:intersections}; for the nodal intersections in the character variety, see Mu\~noz \cite[Theorem 3.1]{Munoz2009}.} Our computation therefore describes the geometry and topology of the singular locus before these conjugacy orbits are collapsed. In general, passage to a character quotient can produce singularities even when the full representation variety is smooth, as can occur for $\mathrm{SL}_2(\C)$-representations of free groups.\fnsep\footnote{For the free group $F_r$ of rank $r$, the representation variety $\Hom(F_r,\SL)\cong(\SL)^r$ is smooth. For $r\ge3$, its character variety $\Hom(F_r,\SL)/\!/\SL$ has a nonempty singular locus consisting exactly of reducible characters; see Florentino and Lawton \cite[Theorem 1.1]{FlorentinoLawton2012}. }

The same integer $\mu=2g$ occurs in the topology of the Milnor fiber $F$ of $z^p+w^q$ at the origin: $F$ has the homotopy type of a wedge of $\mu$ circles.\fnsep\footnote{See Milnor \cite[Theorem 6.5]{Milnor1968}.} We obtain
$b_0(\Sing R)=b_2(\Sing R)=b_1(F)=\mu$.
This equality compares the Betti numbers of the two spaces; its proof uses the explicit geometry of $\Sing R$ in addition to the classical description of $F$.

\section{Components and their intersections}\label{sec:components}

Write $M=\SL$. For $\alpha\in\C^\times\setminus\{1,-1\}$, put $\Cc_\alpha=\{A\in M:\tr A=\alpha+\alpha^{-1}\}$.
Denote the set of commuting pairs in $R$ by
\[
R_0=\{(A,B)\in R:AB=BA\}.
\]
Let $\cI$ consist of the integer pairs $(k,\ell)$ satisfying
\[
1\le k\le p-1,\qquad 1\le\ell\le q-1,
\]
where $k$ and $\ell$ are either both even or both odd. For each $(k,\ell)\in\cI$, set
\[
\alpha_k=e^{\pi i k/p},\qquad
\beta_\ell=e^{\pi i\ell/q},
\]
and put
\[
R_{k,\ell}=\Cc_{\alpha_k}\times\Cc_{\beta_\ell}.
\]
The parity condition ensures that every $(A,B)\in R_{k,\ell}$ satisfies
\[
A^p=B^q=(-1)^kI,
\]
so $R_{k,\ell}\subseteq R$.

\begin{lemma}[Geometry of the components]\label{lem:components}
The irreducible components of $R$ are exactly $R_0$ and the $R_{k,\ell}$ for $(k,\ell)\in\cI$. They satisfy $R_0\cong\SL$, $R_{k,\ell}\cong Q\times Q$, and $|\cI|=g$.
The surface $Q$ and all these components are smooth. The component $R_0$ has dimension three, each $R_{k,\ell}$ has dimension four, and distinct $R_{k,\ell}$ are disjoint.
\end{lemma}

\begin{proof}
We use two sublemmas.

\begin{sublemma}[Geometry of the proposed components]\label{sublem:geometry}
The surface $Q$ is smooth and irreducible. Moreover, $R_0\cong M$ and $R_{k,\ell}\cong Q\times Q$. Thus $R_0$ and every $R_{k,\ell}$ are smooth and irreducible, of dimensions three and four, respectively.
\end{sublemma}

\begin{proof}[Proof of Sublemma~\ref{sublem:geometry}]
The equation $a^2+bc=1$ identifies $Q$ with a smooth affine surface: its gradient $(2a,c,b)$ never vanishes on the surface. Every $U\in Q$ has eigenvalues $1,-1$, and hence is conjugate to $\diag(1,-1)$. Conjugation can be achieved in $M$, since a scalar multiple of a diagonalizing matrix can be chosen to have determinant one. The orbit map from the irreducible variety $M$ onto $Q$ shows that $Q$ is irreducible.

For $\alpha\ne\pm1$, set
\[
c_\alpha=\frac{\alpha+\alpha^{-1}}2,\qquad
d_\alpha=\frac{\alpha-\alpha^{-1}}2.
\]
Since $d_\alpha\ne0$, the mutually inverse maps
\begin{equation}\label{eq:normalization}
Q\longrightarrow\Cc_\alpha,\quad U\longmapsto c_\alpha I+d_\alpha U,
\qquad
A\longmapsto\frac{A-c_\alpha I}{d_\alpha}
\end{equation}
are algebraic isomorphisms. Indeed, the eigenvalues of the image of $U$ are $\alpha,\alpha^{-1}$, and the converse follows from the characteristic polynomial of $A$. Thus every $R_{k,\ell}$ is smooth and irreducible of dimension four.

\Needspace{9\baselineskip}
Choose integers $u,v$ with $uq+vp=1$. The maps
\[
M\longrightarrow R_0,\quad C\longmapsto(C^q,C^p),
\qquad
R_0\longrightarrow M,\quad (A,B)\longmapsto A^uB^v
\]
are inverse morphisms.\fnsep\footnote{Denote the first map by $\phi:M\to R_0$, $\phi(C)=(C^q,C^p)$, and the second by $\psi:R_0\to M$, $\psi(A,B)=A^uB^v$. They are inverse because $\psi(\phi(C))=C$ for every $C\in M$, and $\phi(\psi(A,B))=(A,B)$ for every $(A,B)\in R_0$. These identities follow from $uq+vp=1$, together with $AB=BA$ and $A^p=B^q$. Both maps are morphisms because matrix multiplication and inversion are regular on $M$.} Commutativity and $A^p=B^q$ give
$(A^uB^v)^q=A^{uq+vp}=A$ and $(A^uB^v)^p=B^{uq+vp}=B$.
Negative powers are regular on $M$. Consequently, $R_0\cong M$ is smooth and irreducible of dimension three; it is closed in $R$ because commutativity is a matrix equation.
\end{proof}

\begin{sublemma}[Completeness and counting of the components]\label{sublem:counting}
The varieties $R_0$ and $R_{k,\ell}$, $(k,\ell)\in\cI$, are precisely the irreducible components of $R$. Distinct $R_{k,\ell}$ are disjoint, and $|\cI|=g$.
\end{sublemma}

\begin{proof}[Proof of Sublemma~\ref{sublem:counting}]
For an arbitrary $(A,B)\in R$, put $Z=A^p=B^q$. Both $A$ and $B$ commute with $Z$. If $Z$ is noncentral, its centralizer in $M$ is abelian. To see this, put $Z$ in Jordan form. With distinct eigenvalues its centralizer is a diagonal torus; otherwise the centralizer is conjugate to
\[
\left\{\begin{pmatrix}s&t\\0&s\end{pmatrix}:s\in\{1,-1\},\ t\in\C\right\},
\]
which is also abelian. Hence $AB=BA$ in this case.

If $Z=\varepsilon I$ with $\varepsilon\in\{1,-1\}$, the minimal polynomials of $A$ and $B$ divide $X^p-\varepsilon$ and $X^q-\varepsilon$. These polynomials have simple roots, so $A$ and $B$ are diagonalizable. If either is scalar, the pair lies in $R_0$. Otherwise the eigenvalue pairs are uniquely represented by $\{\alpha_k,\alpha_k^{-1}\}$ and $\{\beta_\ell,\beta_\ell^{-1}\}$ in the indicated ranges. Equality of the powers is exactly the parity condition. We have proved
\begin{equation}\label{eq:components}
R=R_0\cup\bigcup_{(k,\ell)\in\cI}R_{k,\ell}.
\end{equation}

The trace values $2\cos(\pi k/p)$ are pairwise distinct in the stated range, as are the values $2\cos(\pi\ell/q)$. Thus the products are pairwise disjoint closed subvarieties. None is contained in $R_0$, by dimension, and $R_0$ is contained in none of them because $(I,I)$ belongs to $R_0$ and to no product. The finite closed irreducible cover \eqref{eq:components} is therefore the irreducible-component decomposition.

At least one of $p,q$ is odd. If $p$ is odd, exactly $(p-1)/2$ indices $k$ have each parity, so for each $\ell$ that many $k$ are allowed. Thus $|\cI|=(p-1)(q-1)/2=g$. The case where $q$ is odd is symmetric.
\end{proof}

The two sublemmas prove Lemma~\ref{lem:components}.
\end{proof}

\begin{lemma}[Intersections with the abelian component]\label{lem:intersections}
Under the isomorphism $R_{k,\ell}\cong Q\times Q$ in \eqref{eq:normalization}, the reduced intersection with $R_0$ is
\[
R_0\cap R_{k,\ell}\cong\Delta_+\sqcup\Delta_-,
\qquad
\Delta_\pm=\{(U,\pm U):U\in Q\}.
\]
Each $\Delta_\pm$ is closed and isomorphic to $Q$.
\end{lemma}

\begin{proof}
Write $A=c_{\alpha_k}I+d_{\alpha_k}U$ and
$B=c_{\beta_\ell}I+d_{\beta_\ell}V$. Both coefficients $d_{\alpha_k}$ and $d_{\beta_\ell}$ are nonzero. Thus $AB=BA$ if and only if $UV=VU$. In an eigenbasis of $U$, the matrix $U$ is $\diag(1,-1)$, and every commuting matrix is diagonal. Since $V$ has trace zero and satisfies $V^2=I$, it is either $U$ or $-U$. These possibilities are disjoint, because $U=-U$ is incompatible with $U^2=I$ over $\C$. Each is the graph of a regular map $U\mapsto\pm U$, and projection onto the first coordinate is an isomorphism with $Q$.
\end{proof}

\section{The singular locus and its homology}\label{sec:singular}

\begin{proposition}[The singular locus and its profile]\label{prop:singular}
The singular locus is
\begin{equation}\label{eq:singular}
\Sing R=\bigsqcup_{(k,\ell)\in\cI}(R_0\cap R_{k,\ell})
       \cong\bigsqcup_{j=1}^{\mu}Q.
\end{equation}
Equivalently, its points are exactly the pairs $(A,B)\in R$ satisfying
\[
AB=BA,\qquad A,B\notin\{I,-I\},\qquad A^p=B^q\in\{I,-I\}.
\]
In particular, $\dim\Sing R=2$, $P_2(\Sing R)=(\mu,0,0)$, and $\Sing R$ is smooth. Hence $\Sing(\Sing R)=\varnothing$.
\end{proposition}

\begin{proof}
Every component of $R$ is smooth by Lemma~\ref{lem:components}. A point on only one component has a neighborhood in $R$ obtained by deleting all other components. This neighborhood is open in that smooth component, so the point is regular in $R$.

A point on two distinct irreducible components is singular: its local ring has at least two minimal primes, whereas a regular local ring is a domain.\fnsep\footnote{See \cite[Lemma 10.106.2, Tag 00NP]{Stacks}. Over $\C$, regularity and smoothness agree for varieties.} Since distinct four-dimensional components are disjoint, the singular points are precisely the intersections in \eqref{eq:singular}. The matrix characterization follows from their definition and the central-power argument in Lemma~\ref{lem:components}.

By Lemma~\ref{lem:intersections}, each intersection is the disjoint union of two closed copies of $Q$. There are $g$ such intersections. The resulting $2g=\mu$ closed subvarieties are also open in their finite disjoint union, so \eqref{eq:singular} is an isomorphism of reduced varieties. The dimension, profile, smoothness, and emptiness of the next singular locus follow.
\end{proof}

The Andreotti--Frankel--Hamm theorem states that a complex affine variety of complex dimension $n$, possibly singular, has the homotopy type of a CW complex of real dimension at most $n$.\fnsep\footnote{Andreotti and Frankel proved the smooth case \cite{AndreottiFrankel1959}; Hamm treated singular Stein spaces \cite[Satz 1]{Hamm1983}.} Since $\Sing R$ is a closed subvariety of the affine variety $R$, Proposition~\ref{prop:singular} therefore implies
$H_k(\Sing R;\Z)=0$ for $k>2$ and that $H_2(\Sing R;\Z)$ is free abelian. The following computation determines its rank, the lower homology groups, and the homotopy type.

\Needspace{13\baselineskip}
\begin{theorem}[Homotopy type and integral homology]\label{thm:homology}
For relatively prime $p,q>1$,
\begin{equation}\label{eq:homotopy}
\Sing R\simeq\bigsqcup_{j=1}^{\mu}S^2,
\qquad \mu=(p-1)(q-1).
\end{equation}
Consequently,
\begin{equation}\label{eq:homology}
H_k(\Sing R;\Z)\cong
\begin{cases}
\Z^\mu,& k=0,2,\\
0,& k\ge0,\ k\ne0,2.
\end{cases}
\end{equation}
\end{theorem}

\Needspace{9\baselineskip}
\begin{proof}
We give an explicit diffeomorphism from $Q$ to $TS^2$. The linear change of coordinates
\[
z_1=a,\qquad z_2=\frac{b+c}{2},\qquad z_3=\frac{b-c}{2i}
\]
identifies $a^2+bc=1$ with $z_1^2+z_2^2+z_3^2=1$. Write $z=u+iv$, where $u,v\in\RR^3$. This equation becomes
\[
\|u\|^2-\|v\|^2=1,\qquad u\cdot v=0.
\]
Identify $TS^2$ with
$\{(n,v)\in\RR^3\times\RR^3:\|n\|=1,\ n\cdot v=0\}$.
Then
\[
u+iv\longmapsto\left(\frac{u}{\sqrt{1+\|v\|^2}},v\right),
\qquad
(n,v)\longmapsto\sqrt{1+\|v\|^2}\,n+iv
\]
are mutually inverse smooth maps. Contracting each tangent fiber by
$(n,v)\mapsto(n,(1-t)v)$, $0\le t\le1$, gives a strong deformation retraction onto the zero section. Hence $Q\simeq S^2$.

Apply this equivalence separately to the $\mu$ components in \eqref{eq:singular} to obtain \eqref{eq:homotopy}. The singular chain complex of a finite disjoint union is the direct sum of the chain complexes of its components. Homotopy invariance and $H_k(S^2;\Z)\cong\Z$ for $k=0,2$, with zero groups in other degrees, give \eqref{eq:homology}.
\end{proof}

\section{The Milnor fiber and the trefoil}\label{sec:milnor}

Let $f(z,w)=z^p+w^q$. Its link at the origin is the corresponding torus knot in a small three-sphere. The Milnor fibration
\[
\frac{f}{|f|}:S^3_\varepsilon\setminus
             (f^{-1}(0)\cap S^3_\varepsilon)\longrightarrow S^1
\]
has fiber the interior of a compact connected orientable surface $F$ with one boundary component. Equivalently, a compact Milnor fiber is
$f^{-1}(\delta)\cap\overline B^4_\varepsilon$, for suitably small
$0<|\delta|\ll\varepsilon$.\fnsep\footnote{See \cite{Milnor1968}, Chapters 4--5. The interior and compact fiber have the same homotopy type.}

\begin{corollary}[Comparison with the Milnor fiber]\label{cor:milnor}
If $F$ is the Milnor fiber of $z^p+w^q$ at the origin, then
\[
b_0(\Sing R)=b_2(\Sing R)=b_1(F)=\mu=2g.
\]
Thus each nonzero Betti number of $\Sing R$ equals twice the genus of the corresponding torus knot.
\end{corollary}

\begin{proof}
The Milnor number is the dimension of the local Jacobian algebra.\fnsep\footnote{The ring $\C\{z,w\}$ consists of power series $\sum_{i,j\ge0}a_{ij}z^iw^j$, with complex coefficients, that converge in some neighborhood of $(0,0)$. Equivalently, it is the ring of germs of holomorphic functions at the origin: two functions represent the same germ if they agree on some neighborhood of the origin.} Thus
\[
\mu(f)=\dim_\C\frac{\C\{z,w\}}{(\partial f/\partial z,\partial f/\partial w)}
      =\dim_\C\frac{\C\{z,w\}}{(z^{p-1},w^{q-1})}
      =(p-1)(q-1).
\]
The residue classes of $z^iw^j$, $0\le i\le p-2$ and $0\le j\le q-2$, form a basis of the quotient. Milnor's theorem gives $F\simeq\bigvee_{\mu(f)}S^1$, and hence $b_1(F)=\mu(f)$.\fnsep\footnote{See \cite[Theorems 6.5 and 7.2]{Milnor1968} for the bouquet theorem and the identification of the middle Betti number with the multiplicity; see Appendix B, Problem 3, for the local-algebra interpretation.} Combining this with Theorem~\ref{thm:homology} and $\mu=2g$ proves the assertion.
\end{proof}

\paragraph{The trefoil.}
Take $p=2$ and $q=3$. There is one four-dimensional component,
\[
R_{1,1}=\{A\in M:\tr A=0\}\times\{B\in M:\tr B=1\}.
\]
Since $A^2=-I$, its two intersections with $R_0$ are the graphs
\[
B=\frac12 I+\frac{\sqrt3}{2}A,\qquad
B=\frac12 I-\frac{\sqrt3}{2}A
\quad(\tr A=0).
\]
Each graph is isomorphic to $Q$, so
\[
\Sing R\cong Q\sqcup Q\simeq S^2\sqcup S^2.
\]
The first relation is an algebraic isomorphism and the second is a homotopy equivalence. Each component is a smooth complex surface, hence a real four-dimensional manifold, that deformation retracts onto a two-sphere.
In this case $P_4(R)=(1,1,0,0,0)$, $P_2(\Sing R)=(2,0,0)$, and
$H_0(\Sing R;\Z)\cong H_2(\Sing R;\Z)\cong\Z^2$.
The Milnor fiber is a once-punctured torus, with first Betti number two. Figure~\ref{fig:trefoil} compares the two homotopy models.

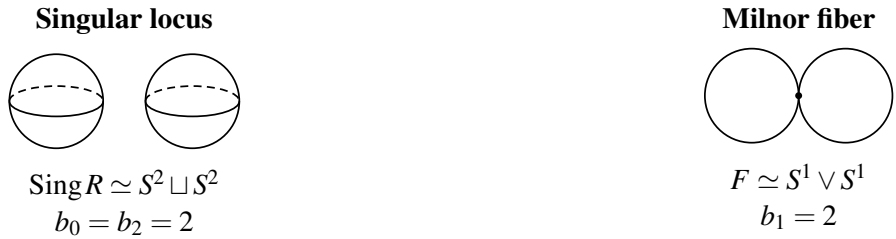
\begin{figure}[htbp]
\centering
\begin{minipage}[t]{0.46\textwidth}
\centering
\small\textbf{Singular locus}\par\smallskip
\begin{tikzpicture}[line width=0.65pt]
  \path[use as bounding box] (-1.8,-0.75) rectangle (1.8,0.75);
  \foreach \x in {-0.9,0.9} {
    \draw (\x,0) circle[radius=0.62];
    \draw[densely dashed] (\x-0.62,0)
      arc[start angle=180,end angle=0,x radius=0.62cm,y radius=0.20cm];
    \draw (\x-0.62,0)
      arc[start angle=180,end angle=360,x radius=0.62cm,y radius=0.20cm];
  }
\end{tikzpicture}\par
$\Sing R\simeq S^2\sqcup S^2$\par
$b_0=b_2=2$
\end{minipage}\hfill
\begin{minipage}[t]{0.46\textwidth}
\centering
\small\textbf{Milnor fiber}\par\smallskip
\begin{tikzpicture}[line width=0.65pt]
  \path[use as bounding box] (-1.8,-0.75) rectangle (1.8,0.75);
  \draw (-0.62,0) circle[radius=0.62];
  \draw (0.62,0) circle[radius=0.62];
  \fill (0,0) circle[radius=1.3pt];
\end{tikzpicture}\par
$F\simeq S^1\vee S^1$\par
$b_1=2$
\end{minipage}
\caption{Homotopy models for the trefoil $(p,q)=(2,3)$, illustrating how the same number $\mu=2$ occurs in different homological degrees.}
\label{fig:trefoil}
\end{figure}

\begin{remark}
If $p=1$ or $q=1$, then $T(q,p)\cong\Z$ and $R\cong\SL$ is smooth. Its singular locus is empty and all ordinary homology groups of that locus vanish. The assertion $\dim\Sing R=2$ therefore requires $p,q>1$.
\end{remark}

\bigskip
\begingroup
\small
\raggedright

\endgroup

\medskip
\noindent\small
S. Liriano\\
Foro Cultural 81\\
Centro Hist\'orico\\
Guanajuato, M\'exico 36000\\
Email: \href{mailto:SAL21458@yahoo.com}{SAL21458@yahoo.com}

\end{document}